\documentclass[11pt]{amsart}

\allowbreak

\usepackage[indentafter]{titlesec}
\titleformat{name=\section}{}{\thetitle.}{0.8em}{\centering\scshape}
\titleformat{name=\subsection}[runin]{}{\thetitle.}{0.5em}{\bfseries}[.\\]
\titleformat{name=\subsubsection}[runin]{}{\thetitle.}{0.5em}{\itshape}[.\\]
\titleformat{name=\paragraph,numberless}[runin]{}{}{0em}{}[.]
\titlespacing{\paragraph}{0em}{0em}{0.5em}
\titleformat{name=\subparagraph,numberless}[runin]{}{}{0em}{}[.]
\titlespacing{\subparagraph}{0em}{0em}{0.5em}

\numberwithin{equation}{section}

\usepackage{amsfonts,amssymb,amscd,amsmath,latexsym,amsbsy}
\usepackage{graphicx}
\usepackage{amssymb}
\usepackage{amsfonts}
\usepackage{latexsym}
\usepackage{mathrsfs}
\usepackage{mathtools}
\usepackage{amsthm}
\usepackage{comment}
\usepackage{listings}
\usepackage{tikz-cd}
\usepackage{titlesec}
\usepackage{hyperref}
\usepackage{amsmath}
\usepackage{leftidx}
\usepackage{centernot}

\allowdisplaybreaks
\newtheorem{thm}{Theorem}[section]
\newtheorem{defn}[thm]{Definition}

\newtheorem{prop}[thm]{Proposition} 
\newtheorem*{rmk*}{Remark} 
\newtheorem{rmk}[thm]{Remark}
\newtheorem{coro}[thm]{Corollary}

\numberwithin{equation}{section}

\newcommand{\R}{\ensuremath{\mathbb{R}}}

\newcommand{\N}{\ensuremath{\mathbb{N}}}

\newcommand{\mk}{\ensuremath{\mathfrak}}

\newcommand{\la}{\ensuremath{\langle}}
\newcommand{\ra}{\ensuremath{\rangle}}

\title{The Topological Pressure of Trapped Sets in Kerr-(de Sitter) Spacetimes}

\author{Qiuye Jia}
\date{\today}

\begin{document}

\begin{abstract}
In this paper we prove that the topological pressure of dynamical systems with normally hyperbolic trapping is negative. In particular, this applies to the null geodesic flow in Kerr and Kerr-de Sitter spacetimes.

This builds a connection between results for trapped sets with low regularity in hyperbolic dynamical systems conditioning on negativity of the topological pressure and unconditional results in the setting of normally hyperbolic trapping.
\end{abstract}

\maketitle

\section{Introduction}

\subsection{The main result}
Let $\varphi^{\textbf{t}}$ be a flow (i.e., a one-parameter family of smooth maps) on a smooth manifold $\mathscr{M}$  and we will study the topological pressure, a quantity characterizing the `thickness' of its trapped set, associated with it.

Define the backward/forward trapped set by:
\begin{align}  \label{eq: definition Gamma pm}
\Gamma_{\pm} := \{  p \in \mathscr{M} :    \varphi^{\textbf{t}}(p) \centernot\longrightarrow \infty \text{ when } \textbf{t} \rightarrow \mp \infty \},
\end{align}
where $\centernot\longrightarrow \infty$ means remaining in a compact subset of $\mathscr{M}$. Then
\begin{align} \label{eq: definition, Gamma}
\Gamma = \Gamma_+ \cap \Gamma_-
\end{align}
is the trapped set.

Following \cite[Definition~20.2.1]{katok1995introduction} (see also \cite[Section~3.1]{burq2010strichartz} and \cite[Section~3.3]{nonnenmacher-zworski2009quantum}), the topological pressure of a flow on its trapped set is defined as follows. Let $d$ be the metric induced from the fiber norm of $T\mathscr{M}$ in Definition \ref{definition: r- NH trapping} below. A set $E \subset \Gamma$ is said to be $(\epsilon,T)-$separated if for any $p_1,p_2 \in E$, there exists $
\textbf{t} \in [0,T]$ such that 
\begin{align*}
d(\varphi^{\textbf{t}}(p_1),\varphi^{\textbf{t}}(p_2)) \geq \epsilon.
\end{align*}
Let $\lambda^u_1$ be the logarithmic unstable Jacobian defined in Section \ref{sec: log unstable Jacobian}, after we characterize the stable-unstable splitting of the dynamical system. Then we define 
\begin{align}
Z_T(\epsilon,s):= \sup_E \sum_{p \in E} e^{-s\lambda^u_1(p)},
\end{align}
where the $\sup$ is taken over all $(\epsilon,T)-$separated sets. Then we define
the topological pressure to be
\begin{align} \label{eq: definition of TPP, separated set} 
P_\Gamma(s):=\lim_{\epsilon \rightarrow 0} \limsup_{T \rightarrow \infty} \frac{1}{T} \log Z_T(\epsilon,s).
\end{align}
By \cite[Corollary~9.10.1]{walters2000introduction} (see also \cite[Equation~(1.4)]{nonnenmacher2015decay} in this setting of normally hyperbolic trapping), we have the following variational principle characterization of the topological pressure:
\begin{align} \label{eq: definition of TPP, variational}
P_{\Gamma}(s) = \sup_{\mu \in \mathrm{Erg}(\Gamma)} (h_\mu(\varphi^1)-s \int_\Gamma \lambda^u_1 d\mu)
\end{align}
where $\mathrm{Erg}(\Gamma)$ is the set of all flow-invariant ergodic measures supported on $\Gamma$, $h_{\mu}(\varphi)$ is the Kolmogorov-Sinai entropy of the measure $\mu$, and $\lambda^u_1$ is again the logarithmic unstable Jacobian defined in Section \ref{sec: log unstable Jacobian}. 

In order to give readers a sense of what this pressure is characterizing, we mention without proof its relationship with $d_\Gamma$, the Hausdorff dimension of the trapped set, in certain special cases.

Let $\mathscr{M}=S^*X$ be the cosphere bundle of a 2-dimensional Riemannian manifold $(X,g)$, and we take $\varphi^{\textbf{t}}$ to be the geodesic flow
of $g$. When $\varphi^{\textbf{t}}$ is hyperbolic on $\Gamma$,  
\begin{align*}
P_\Gamma(\frac{1}{2})<0,
\end{align*}
 is equivalent to 
\begin{align*}
d_\Gamma<2.
\end{align*}
See the discussion after \cite[Theorem~3]{nonnenmacher-zworski2009quantum} and \cite[Remark~3.4]{burq2010strichartz}.

Another example concerns the convex cocompact hyperbolic manifolds, denoted by $X$, and again we take $\mathscr{M}=S^*X$.  Here $X$ is the quotient by the group action on the hyperbolic space $\mathbb{H}^{n+1}$ (already conformally compactified) of a discrete subgroup $G$ of $SO(n,1)$. 
Convex cocompactness means the projection of the convex hull of the limiting set from $\mathbb{H}^{n+1}$ to its quotient is compact.

By the result of \cite{sullivan1979density}\cite{zworski1999dimension}, suppose the limiting set in $X$ has Hausdorff dimension $\delta \in [0,n)$, then the trapped set $\Gamma \subset S^*X$ has Hausdorff dimension 
\begin{align*}
d_\Gamma = 2\delta+1.
\end{align*}
On the other hand, \cite[Lemma~3.5]{burq2010strichartz} (which used the result of \cite{sullivan1979density}, saying that $P_\Gamma(0)=\delta$) shows that with assumptions above, the topological pressure of its geodesic flow is
\begin{align*}
P_\Gamma(\frac{1}{2}) = \delta - \frac{n}{2}.
\end{align*}
Then the topological pressure condition $P_\Gamma(\frac{1}{2})<0$ is equivalent to 
\begin{align*}
d_\Gamma < n+1.
\end{align*}

Our main result is that the topological pressure associated with the dynamical system with normally hyperbolic trapping is negative. 
\begin{thm} \label{thm: main}
Suppose $\Gamma$ is the trapped set of a flow with $\infty-$normally hyperbolic trapping, then $P_{\Gamma}(0)=0$, and
\begin{align}
P_\Gamma(s)<0,
\end{align}
for any $s>0$. In particular, $P_\Gamma(\frac{1}{2})<0$.
\end{thm}

We will show that the trapping of the null geodesic flow in the Kerr and Kerr-de Sitter spacetimes is $\infty-$normally hyperbolic, hence it has negative topological pressure.

The dynamical behaviour of the geodesic flow has its own geometric interest, but we want to emphasize how it will affect the analysis of partial differential equations on manifolds.
In general, suppose $P$ is a pseudodifferential operator on a manifold $M$ with homogeneous principal symbol $p$, which is a function on $T^*M$. We can define its Hamilton vector field $H_p$ using the natural symplectic structure on $T^*M$. Then the result of H\"ormander\cite{hormander_propagation} and Duistermaat-H\"ormander\cite{duistermaat1972fourier} shows that, assuming that $f$ is smooth (the cited results deal with $f$ of general regularity, which we omit for simplicity)
and the $H_p$-flow is non-trapping, then the solution to
\begin{align}
Pu = f,
\end{align}
will remain the same regularity along the $H_p-$flow. This type of phenomenon is called the \emph{`propagation of singularities'}.

In particular, suppose $M$ is equipped with a metric $g$ and denote the fiber variable in $T^*M$ by $\xi$, then the principal symbol of the d'Alembert operator $P=\Box_g$ is $p = |\xi|^2_{g^{-1}}$, and the $H_p-$flow is the geodesic flow. The propagation phenomena indicates that the regularity of solutions to the wave equation on this spacetime, which in turn is closely related to the linearized Einstein equation, is closely related to the $H_p-$flow.

However, the aforementioned propagation of singularities in \cite{hormander_propagation}\cite{duistermaat1972fourier} only deals with the case where the $H_p-$flow is non-trapping since it only predicts regularity along the $H_p-$flow. Thus efforts have been spent to overcome this issue under certain conditions on the trapping. In particular, in the setting of the normally hyperbolic trapping, a resolvent estimate was obtained by Wunsch and Zworski \cite{wunsch2011resolvent}, which is applied to analyze the wave equations and their perturbations on Kerr-de Sitter spacetimes by Vasy \cite{vasy2013microlocal}. A spectral gap result was proven by Dyatlov \cite{dyatlov2016spectral}, which essentially initiated the idea of proving a propagation type estimate by a two step propagation process. This two step propagation was generalized to the time dependent setting by Hintz \cite{hintz2021normally}. Then the author \cite{jia2022propagation} improved the loss of regularity in the propagation estimate 
using a new pseudodifferential algebra, which is defined by introducing a resolution of singularity on the unstable manifold of the dynamical system of the null geodesic flow.

Since the topological pressure is a key feature of the `thickness' of the trapped set, hence characterizing the intensity of the trapping phenomena, its negativity naturally arises as conditions of various estimates.

When the dynamical system is hyperbolic (in the sense that $T\mathscr{M}$ restricted to $\Gamma$ splits into three parts: the flow direction, the unstable bundle and the stable bundle) a spectral gap result and a resolvent bound are proved by Nonnenmacher and Zworski\cite[Theorem~3 and~5]{nonnenmacher-zworski2009quantum} for semiclassical Schr\"odinger operators, conditioning on $P_\Gamma(\frac{1}{2})<0$. In \cite{burq2010strichartz}, under similar dynamical assumptions, Burq, Guillarmou and Hassell proved a Strichartz estimate for the Schr\"odinger propagator, also conditioning on $P_\Gamma(\frac{1}{2})<0$.
These results conditioning on $P_\Gamma(\frac{1}{2})<0$ are the major motivation of our current paper. Although they have different dynamical assumptions, i.e., normally hyperbolic instead of hyperbolic, it still could be considered as an indication that the normally hyperbolic trapping is weak in terms of its effect on the propagation of singularities. 

To illustrate the similarity of analytical aspects of these two types of dynamical systems, we introduce the logarithmic Sobolev order first. 
Let $\la \cdot \ra$ denote $(1+|\cdot|^2)^{1/2}$. We use 
\begin{align*}
||\la \xi \ra^s \hat{u}||_{L^2},
\end{align*} 
as the $s-$order Sobolev norm of $u$, where $\hat{u}$ is the Fourier transform of $u$ and $\xi$ is the dual variable in $\hat{u}$.

Then we introduce the logarithmic order by adding a power of $\la \log (1+|\xi|) \ra$ in the definition of norm. Precisely, the space equipped with the norm (and functions are defined to be in this space if this norm is finite)
\begin{align*}
||\la \log (1+|\xi|) \ra^m \la \xi \ra^s u||_{L^2},
\end{align*}
is called Sobolev space with $s-$polynomial order and $m-$logarithmic order. The propagation type estimate in the non-trapping case for $k-$th order pseudodifferential operator $P$ takes the form 
\begin{align*}
||Qu||_{H^s} \leq ||Q'u||_{H^s} + ||GPu||_{H^{s-k+1}}+||u||_{H^{-N}},
\end{align*}
where $Q,Q',G \in \Psi^0$ are microlocalizers that we are not going to describe in detail. In the case with trapping, this fails and we have `the loss of regularity' in the sense that $||GPu||_{H^{s-k+1}}$ needs to be replaced by stronger Sobolev norms. The difference between the norm needed in the estimate and this optimal non-trapping estimate is called \emph{the loss of regularity}.

The result in \cite{nonnenmacher-zworski2009quantum} shows that conditioning on $P_{\Gamma}(1/2)<0$, the propagation estimate for operator $P$ with hyperbolic bicharacteristic flow (i.e., the $H_p-$flow above) has logarithmic loss. On the other hand, propagation estimates with normally hyperbolic trapping also have logarithmic loss and are unconditional, as proven in \cite{wunsch2011resolvent}\cite{nonnenmacher2015decay} in the time-independent case, and indicated by the arbitrarily small polynomial loss proven in \cite{jia2022propagation} in the time-dependent case.
Thus, this motivates us to verify this topological pressure condition in the normally hyperbolic setting to give an explanation that we do not need to make this assumption: it is automatically satisfied.

This paper is organized as follows. In Section \ref{sec: dynamic preliminaries}, we introduce basic concepts, assumptions and results we need in the theory of dynamical systems. We will first introduce the precise definition of the $\infty-$normally hyperbolic trapping, and the logarithmic unstable Jacobian. Then we will show that the topological entropy term in the variational characterization of the topological pressure will vanish for dynamical systems with $\infty-$normally hyperbolic trapping. Then in Section \ref{sec: proof of the main theorem}, we prove Theorem \ref{thm: main}. 
Finally, in Section \ref{sec: application,KdS}, we briefly introduce Kerr(-de Sitter) spacetimes and show that Theorem \ref{thm: main} applies to the null geodesic flow on Kerr-(de Sitter) spacetimes.

\section{Dynamical preliminaries} \label{sec: dynamic preliminaries}

\subsection{\texorpdfstring{$\infty-$}{infinity }normally hyperbolic trapping}
We first recall the definition of eventually absolutely $r$-normally hyperbolic trapping in the sense of \cite{hirsch2006invariant} (see also \cite[Section~3.2]{dy15}\cite[Section~1.2]{wunsch2011resolvent}). 

\begin{defn}   \label{definition: r- NH trapping}
We say the trapping of $\varphi^{\textbf{t}}$ is eventually absolutely $r$-normally hyperbolic trapping for $r \in \N$, if there is a splitting of $T\mathscr{M}$ over its trapped set $\Gamma$:
\begin{align*}
T_\Gamma \Gamma^{u/s} = T\, \Gamma \oplus E^{u/s},\\
d\varphi^{\textbf{t}} : \, E^{u/s} \rightarrow E^{u/s},
\end{align*}
and there exists a fiber norm $|\cdot|$ and $\nu>0$ such that
\begin{align} \label{eq: exponential property of the flow}
|d\varphi^{\textbf{t}}(z)v| \leq C e^{-\nu|\textbf{t}|} |v|,  \begin{cases}
v \in E^u(z),\quad \textbf{t} \leq 0 \\
v \in E^s(z), \quad \textbf{t} \geq 0.
\end{cases}
\end{align}
And in addition, there exist $\theta_0,C>0$, such that for $\textbf{t}>0$:
\begin{align} \label{eq:r-NH}
\begin{split}
& \sup_{T\,\Gamma}||d\varphi^{\textbf{t}}|_{T\,\Gamma}||^r \leq Ce^{-\textbf{t}\theta_0}\inf_{T\,\Gamma}||d\varphi^{-\textbf{t}}|_{E_u}||^{-1},\\
& \inf_{T\,\Gamma}||d\varphi^{-\textbf{t}}|_{T\,\Gamma}||^{-r} \geq C^{-1}e^{\textbf{t}\theta_0} \sup_{T\,\Gamma}||d\varphi^{\textbf{t}}|_{E_s}||.
\end{split}
\end{align}
\end{defn}

A formulation that is equivalent and more convenient for us is, let $\nu_{\min}$ be the supremum of all $\nu$ such that (\ref{eq: exponential property of the flow}) holds, and let $\mu_{\max}$ be the maximal (asymptotic) expansion rate along $T\Gamma$, then 
\begin{align}  \label{eq: expansion rate, r-NH trapping}
\nu_{\min} > r \mu_{\max}.
\end{align}

Intuitively, (\ref{eq: exponential property of the flow}) means that the flow is hyperbolic in the $E^{u/s}-$directions, which are normal to $\Gamma$. And the additional assumption \eqref{eq:r-NH} means that the expansion/contraction on $E^{u/s}$ is considerably stronger than any expansion and contraction occurring in the flow on $\Gamma$. Now the $\infty-$normally hyperbolic trapping is defined as:
\begin{defn} \label{def:infty-NH-trapping}
If the trapping for $\varphi^{\textbf{t}}$ is eventually absolutely $r$
-normally hyperbolic trapping for all $r \in \N$, then we say that it is $\infty-$normally hyperbolic trapping.
\end{defn}

\subsection{The logarithmic unstable Jacobian} \label{sec: log unstable Jacobian}
In this section we define the logarithmic unstable Jacobian, which completes the definition of the topological pressure in (\ref{eq: definition of TPP, separated set}) and (\ref{eq: definition of TPP, variational}).

The fiber metric in Definition \ref{definition: r- NH trapping} induces a volume form $\Omega$ on $E^u$ (in fact also on other subbundles of $T\mathscr{M}$), and we denote its value at $p$ by $\Omega_p$. Denote the dimension of fibers of $E^u$ by $n_u$, and let $v_1,...,v_{n_u}$ be a basis of $E^u_p$ at a point $p \in \Gamma$, then we define the $\Omega-$determinant of $\varphi^{\textbf{t}}$  to be
\begin{align}
\det(d\varphi^{\textbf{t}}|_{E^{u}_p}):=
\frac{\Omega_{\varphi^t(p)}( d\varphi^{\textbf{t}}v_1 \wedge...\wedge d\varphi^{\textbf{t}}v_{n_u} )}{
	\Omega_p(v_1 \wedge ... \wedge v_{n_u})}.
\end{align}

As pointed out after \cite[Equation~(3.19)]{nonnenmacher-zworski2009quantum}, the resulting topological pressure is independent of the choice of the volume form (and in turn, independent of the fiber norm in (\ref{eq: exponential property of the flow})), hence we have omitted $\Omega$ in the notation. 

Using the determinant defined above, the logarithmic unstable Jacobian is defined to be
\begin{align}
\lambda^u_k(p) = \log ( \det(d\varphi^{\textbf{t}}|_{E^{u}_p})|_{\textbf{t}=k})
\end{align}

\subsection{Lyapunov exponents and the entropy}

\begin{prop} \label{prop: Lyapunov exp is 0, infinite normally hyperbolic trapping}
All Lyapunov exponents of the flow restricted to $T\Gamma$ for $\varphi^{\textbf{t}}$ which is $\infty-$normally hyperbolic trapping can only be 0.
\end{prop}

\begin{proof}
This follows from \eqref{eq: expansion rate, r-NH trapping} together with consideration of both the forward and backward flows.
\end{proof}


The tool we need to build the connection between the dynamical behaviour and the entropy term in (\ref{eq: definition of TPP, variational}) is the Margulis-Ruelle inequality, which we recall next and we refer the reader to \cite{Ruelle-inequality} for more details.
\begin{thm} \label{thm:Ruelle-inq}
Let $\varphi$ be a $C^{1+\epsilon}-$diffeomorphism $\Gamma \rightarrow \Gamma$ and let $\mu$ be 
a $\varphi-$invariant probability measure, then we have
\begin{align} \label{eq: Ruelle-ineq}
h_{\mu}(\varphi) \leq \int_\Gamma  \Sigma(x) d\mu(x),
\end{align}
where $\Sigma(x)$ is the sum of all positive Lyapunov exponents:
\begin{align*}
\sum_{\chi_i(x)>0} \chi_i(x)k_i(x),
\end{align*}
with $\chi_i(x)$ being the positive Lyapunov exponents of $\varphi$, and $k_i$ being its multiplicity.
\end{thm}

\begin{rmk}
When $\varphi$ is absolutely continuous (with respect to the Lebesgue measure of the background manifold), \eqref{eq: Ruelle-ineq} becomes an equality, which is called the Pesin entropy formula (see \cite[Section~5]{pesin1977entropy}). But in our case, it is important to allow $\mu$ that are not absolutely continuous, hence we need to use this inequality version. 
\end{rmk}

Since the Kolmogorov-Sinai entropy is non-negative by its definition, combining Proposition \ref{prop: Lyapunov exp is 0, infinite normally hyperbolic trapping} and Theorem \ref{thm:Ruelle-inq}, we have:

\begin{coro} \label{coro: entrop vanishes}
The Kolmogorov-Sinai entropy of $\varphi^1|_\Gamma$ vanishes:
\begin{align}
h_{\mu}(\varphi^1) = 0.
\end{align}
\end{coro}

\section{Proof of the main theorem} \label{sec: proof of the main theorem}
In this section we prove Theorem \ref{thm: main}. Combining (\ref{eq: definition of TPP, variational}) and Corollary \ref{coro: entrop vanishes}, we know for $s\geq 0$,
\begin{align}  \label{eq: pressure, only 2nd term}
P_{\Gamma}(s) = s\sup_{\mu \in \mathrm{Erg}(\Gamma)} (- \int_\Gamma \lambda^u_1 d\mu),
\end{align}
which gives $P_\Gamma(0)=0$. Next we consider the case $s>0$.

Since the measure $\mu$ is flow invariant, the integral
\begin{align*}
\int \lambda^u_1(p) d\mu,
\end{align*}
equals the integral of

\begin{align*}
\log (\frac{\Omega_{\varphi^{k+1}(p)}( d\varphi^{k+1}v_1 \wedge...\wedge d\varphi^{k+1}v_{n_u} )}{\Omega_{\varphi^{k}(p)}(d\varphi^{k}v_1 \wedge ... \wedge d\varphi^{k}v_{n_u})}),
\end{align*}
which in turn equals
\begin{align*}
\int (\lambda^u_{k+1}(p)-\lambda^u_k(p)) d\mu.
\end{align*}
Taking the sum over $k$, we know 
\begin{align*}
\int \lambda^u_1 d\mu = \int \frac{1}{k} \lambda^u_k d\mu .
\end{align*}
Considering the integrand 
\begin{align}
\frac{1}{k} \lambda^u_k = \frac{1}{k} \log (\frac{\Omega_{\varphi^{k+1}(p)}( d\varphi^{k+1}v_1 \wedge...\wedge d\varphi^{k+1}v_{n_u} )}{\Omega_p(v_1 \wedge ... \wedge v_{n_u})}),
\end{align}
by (\ref{eq: exponential property of the flow}) we know for any $\epsilon>0$, this quantity is at least $n_u\nu-\epsilon$ when $k$ is large enough. Hence by (\ref{eq: pressure, only 2nd term}) we know $P_{\Gamma}(s)<0$ for $s>0$.

\begin{rmk}
In the proof above, we have used the flow-invariant property to replace the integrand by the logarithmic unstable Jacobian of a large time, which is strictly positive due to the asymptotic behaviour of the flow. An alternative proof can be given using the adapted metric $g_{\mathrm{ad}}$, that is the fiber metric such that we can take $C=1$ in (\ref{eq: exponential property of the flow}), which is similar to the requirement in (3.12,vi) of \cite{nonnenmacher-zworski2009quantum} in the hyperbolic setting, and it exists due to \cite[Theorem~4]{gourmelon2007adapted}.
After resorting to $g_{\mathrm{ad}}$, the time 1 logarithmic unstable Jacobian is strictly positive as well, and we can obtain the result directly.
\end{rmk}

\begin{rmk} \label{remark:spectral-gap-relation}
We notice that the bound on the topological pressure produced by the proof above is proportional to the expansion or contraction rate $\nu$ in \eqref{eq: exponential property of the flow}, which in turn is proportional to the spectral gap proved in \cite[Theorem~3]{dy15}. See also \cite{dyatlov2016spectral}\cite{dyatlov2015resonance}.

This expansion rate enters the spectral gap in the following way. 
The spectral gap is proven via a resolvent estimate, which in turn comes from a positive commutator argument.
See for example \cite[Section~2.3, Section~3]{dyatlov2016spectral}\cite[Section~8]{dyatlov2015resonance}.
Then the imaginary part of the spectral parameter becomes the skew-symmetric part of the commutator. 
In this positive commutator argument, in order to propagate control in the correct direction, we want a favored sign of this part combined with the term introduced by differentiating the commutant along the Hamilton flow (of the symbol of our operator), which is the expansion rate. Consequently, this becomes an inequality between the imaginary part of the spectral parameter and the expansion rate, which leads to the spectral gap.
\end{rmk}

\section{Application} \label{sec: application,KdS}
The major application of Theorem \ref{thm: main} is to the null geodesic flow on Kerr and Kerr-de Sitter spacetimes, which model rotating black holes. Specifically, the metric below solves the Einstein vacuum equation with cosmological constant $\Lambda=0$ and $\Lambda>0$ respectively. Since their difference lies near the cosmological horizon, which is not our major concern, we use Kerr(-de Sitter) to refer to both of them below, and our results will be uniform as $\Lambda \rightarrow 0$. Mathematically, a Kerr(-de Sitter) spacetime is a manifold $M^\circ$ equipped with a Lorentzian metric $g_{\mk{m},\mk{a}}$ such that
\begin{align}
M^\circ=\R_t \times X , \quad X=(r_e,r_c) \times \mathbb{S}^2,
\label{kds_manifold}
\end{align}
where $r_e,r_c$ are defined in (\ref{eq: KdS, 4 roots}) below. The Lorentzian metric $g_{\mk{m},\mk{a}}$ is determined by two parameters: the angular momentum $\mk{a}$ and the mass $\mk{m}$:
\begin{align}
\begin{split}
g_{\mk{m},\mk{a}} =& (r^2+\mk{a}^2\cos^2\theta)(\frac{dr^2}{\Delta(r)}+\frac{d\theta^2}{\Delta_\theta})+\frac{\Delta_\theta \sin^2\theta}{\Delta_0^2(r^2+\mk{a}^2\cos^2\theta)}(\mk{a}dt-(r^2+\mk{a}^2)d\varphi)^2\\
&-\frac{\Delta(r)}{\Delta_0^2(r^2+\mk{a}^2\cos^2\theta)}(dt-\mk{a}\sin^2\theta d\varphi)^2,
\end{split}
\label{kds_metric}
\end{align}
where $\Lambda$ is the cosmological constant and
\begin{align}
\begin{split}
&\Delta(r) = (r^2+\mk{a}^2)(1-\frac{\Lambda r^2}{3})-2\mk{m}r, \quad \Delta_\theta=1+\frac{\Lambda \mk{a}^2}{3}\cos^2\theta,\\
&\Delta_0 = 1+\frac{\Lambda \mk{a}^2}{3}, \quad \Lambda \geq 0.
\end{split}
\label{delta_intro}
\end{align}
We assume the black hole is subextremal in the sense that
\begin{align} \label{eq:subextremal-condition}
\mathfrak{a}<\mathfrak{m}, \quad  \Lambda \mk{m}^2<\tilde{\Lambda}_0,
\end{align}
where $\tilde{\Lambda}_0$ is the root (in $\Lambda \mk{m}^2$) of the determinant of $\Delta(r)$ as a quartic equation in $r$ when $\mk{a}=\mk{m}$. Numerically, by \cite[Equation~(30)]{sarp2011kds}, we have $\tilde{\Lambda}_0 \approx 0.1528$.
Conditions in \eqref{eq:subextremal-condition} imply that (see \cite[Section~3]{sarp2011kds} for more details)
\begin{align*}
\Delta(r) = (r^2+\mk{a}^2)(1-\frac{\Lambda r^2}{3})-2\mk{m}r
\end{align*}
has four distinct real roots
\begin{align} \label{eq: KdS, 4 roots}
r_-<r_C<r_e<r_c,
\end{align}
for $\Lambda>0$, and $r_c \rightarrow \infty$ as $\Lambda \rightarrow 0$, hence we take $r_c = \infty$ in that case. 
Consequently, the range defined by \eqref{eq:subextremal-condition} is included in the range allowed in \cite{petersen2022wave}.
Physically, $r_e$ is the location of the event horizon, and $r_c$ is the location of the cosmological horizon. The fact $r_c = \infty$ when $\Lambda=0$ corresponds to the fact that there is no cosmological horizon, or the cosmological horizon is `at infinity', in the  model with $\Lambda=0$.

The metric $g_{\mk{m},\mk{a}}$ defines a dual metric function $G(z,\zeta):=|\zeta|^2_{g_{\mk{m},\mk{a}}^{-1}(z)}$ on $T^*M^\circ$, and the flow induced by its Hamilton vector field $H_G$, defined using the natural symplectic structure on $T^*M^\circ$, is the geodesic flow (we are identifying $T^*M^\circ$ and $TM^\circ$ using this metric structure). We define $\varphi^{\textbf{t}}$ to be the flow induced by this vector field, and take $\mathscr{M}=T^*M^\circ$.

\begin{prop} \label{prop:Kerr-Kds}
The topological pressure of $\varphi^{\textbf{t}}$ defined using the flow of $H_G$ on subextremal Kerr and Kerr-de Sitter spacetimes is negative when $s>0$.
\end{prop}
\begin{proof}
By Theorem \ref{thm: main}, the result follows if we know that the trapping of $\varphi^{\textbf{t}}$ on Kerr and Kerr-de Sitter spacetimes is $\infty-$normally hyperbolic trapping. On Kerr spacetimes, this is proven in \cite[Proposition~3.6, 3.7]{dy15}. 

On Kerr-de Sitter spacetimes, the hyperbolicity is verified in \cite{vasy2013microlocal} for the range with $\mk{a}<\frac{\sqrt{3}}{2}\mk{m}$, and then extended to the entire subextremal range $\mk{a}<\mk{m}$ in \cite[Theorem~3.2]{petersen2022wave}, and now we explain how 
the proof of \cite[Proposition~3.6, 3.7]{dy15} can be adapted to this setting in combination with \cite[Theorem~3.2]{petersen2022wave}, and this finishes the proof.

The ingredients in the proof of \cite[Proposition~3.6, 3.7]{dy15} that need a justification when we extend to region \eqref{eq:subextremal-condition} are the characterization of the trapped set $\Gamma$ and the backward/forward trapped set $\Gamma_\pm$, and \cite[Equation~(3.19)]{dy15}.
Firstly, the characterization of the trapped set $\Gamma$ and the backward/forward trapped set $\Gamma_\pm$ in \cite[Proposition~3.3, 3.5]{dy15} extends to the full subextremal range of Kerr-de Sitter spacetimes by \cite[Theorem~3.2,(c)(d)]{petersen2022wave} since those defining functions combined together are equivalent.
As mentioned after \eqref{eq: KdS, 4 roots}, this range includes the range we defined by \eqref{eq:subextremal-condition}.
In addition, \cite[Equation~(3.19)]{dy15} extends to our region by the second part of \cite[Theorem~3.2.(a)]{petersen2022wave}.
Other steps in the proof of \cite[Proposition~3.6, 3.7]{dy15} extend to our range \eqref{eq:subextremal-condition} directly.

\end{proof}

We conclude by discussing the range of parameters allowed and the extremal case. 
We expect Proposition~\ref{prop:Kerr-Kds} extends to a range of parameters that is larger than \eqref{eq:subextremal-condition}. Because we only used \cite[Theorem~3.2]{petersen2022wave} to justify requirements in the proof of \cite[Proposition~3.6, 3.7]{dy15}, the argument we gave should go through in the range that only requiring $\Delta(r)$ to have four distinct real roots, which is the range used in \cite{petersen2022wave}. We did not choose to do so because there is no very simple or explicit upper bound of $\frac{\mk{a}}{\mk{m}}$ and $\Lambda \mk{m}^2$ to characterize this range.

When we are indeed in the extremal case, two roots of $\Delta(r)$ merge together and the structure of the spacetime changes. If we increase $\Lambda\mk{m}^2$ so that $r_e$ and $r_c$ merge, then the interval $(r_e,r_c)$ on which the analysis in \cite{petersen2022wave} happens collapses to a point and the characterization of the trapped set is not valid anymore.
Potentially, one can continue to use similar algebraic equations to define $\Gamma$ and $\Gamma_\pm$, but there are additional issues to deal with.
For example, the dynamics in the normal directions is not strictly hyperbolic anymore.
Using the expression of the linearization of the null bicharacteristic flow in \cite[Theorem~3.2.(e)]{petersen2022wave}, after taking $r_{\xi_t,\xi_{\phi}}$ there to be $r_e=r_c$, the matrix in the leading order term is not strictly hyperbolic anymore since $\mu(r_e)=0$.

If we increase $\mk{a}$ so that $r_C$ and $r_e$ merge. Consider the case $\Lambda=0$ and $\mk{a}=\mk{m}$ for simplicity. Then the analysis in \cite[Proposition~3.9]{dy15} (in particular, Eq.(3.38) there) shows that the dynamics in the normal directions is not strictly hyperbolic anymore and the conclusion fails.

\section*{Acknowledgements}
The author would like to thank Andrew Hassell for asking the question that is equivalent to the main theorem here. The author is very grateful to Zhongkai Tao, who told the author about tools needed in the theory of dynamical systems.
The author is also grateful for the helpful comments of the referee. In particular, the referee's suggestion regarding the correction of the form of Theorem~\ref{thm:Ruelle-inq} and the connection between the spectral gap and the topological pressure discussed in Remark~\ref{remark:spectral-gap-relation}.

\bibliographystyle{plain}
\bibliography{bib_dynamic}

\end{document}